\documentclass[12pt]{elsarticle}\journal{xxx}
\usepackage{amssymb,amsmath,amsthm,bm}
\usepackage{amsfonts}
\usepackage{epsfig}
\usepackage{amsbsy}
\usepackage{color}
\usepackage{epstopdf}
\usepackage{subfigure}
\usepackage{float}
\usepackage{geometry}
\biboptions{sort&compress}

\usepackage{graphicx}
\graphicspath{{figures/}}

\newlength\imagewidth
\newtheorem{Theorem}{Theorem}

\theoremstyle{definition} \newtheorem{Example}{Example}

\usepackage{color}
\definecolor{dgreen}{rgb}{0,.6,0}

\usepackage[normalem]{ulem}
\begin{document}

\begin{frontmatter}
\title{
Modelling the inhibiting effect on a microbial pesticide model}

\author{Xiaoxiao Cui
\,\,\,\, Yonghui Xia\footnote{Author for correspondence. Yonghui Xia, ORCID: 0000-0001-8918-3509.
Email: xiadoc@163.com; yhxia@zjnu.cn. This work was supported by the National Natural
Science Foundation of China under Grant (No. 11931016, 11671176).}
}
\address{ College of Mathematics and Computer Science, \\
Zhejiang Normal University, Jinhua, 321004, China}


\begin{abstract}
Microbial pesticides can avoid many of negative effects of traditional chemical pesticides. To modelling the inhibiting effect, in this paper we propose a model of entomopathogenic nematodes killing the target insects, and inhibiting the birth rate of the target insects simultaneously. In the model, we achieve the purpose of restricting or eliminating pests by continuously releasing nematodes. By analyzing the stability of the equilibrium of the model and the stability of the Hopf bifurcation periodic solution, the best solution to control pests is obtained, and our conclusions are verified by examples and numerical simulations.\vspace{0.3cm}
\vspace{0.3cm}{\bf 2020 Mathematics Subject Classification.} 92D25; 34D20; 37G15.\\

\end{abstract}
\end{frontmatter}

\vspace{1cm}
\section{\Large{Introduction}}

\par\noindent

\subsection{Research motivation}

Pesticides can control agricultural pests and increase food production, but they also bring many disadvantages \cite{Ionel,Mikhail,James}. After pesticides are applied, part of them are adhere to plants, or penetrate into the plant body and remain, contaminating grains, vegetables, fruits, etc., and the other part of them are scatter on the soil or evaporate, escape into the air, or flow into the rivers with rainwater, polluting water bodies and aquatic organisms, and eventually entering the human body, causing various chronic or acute diseases. The unreasonable use of pesticides, especially organic pesticides, not only poses a serious threat to human health, but also causes crop phytotoxicity, human and livestock poisoning, excessive residues of agricultural products, pest resistance and environmental pollution, etc.


Microbial pesticides are made from living microorganisms. In nature, there are many microorganisms that have pathogenic effects on pests, and using this pathogenicity to control pests is an effective biological control method. From these pathogenic microorganisms, select bacteria that are convenient to use, stable in efficacy, safe to humans, animals, and the environment to make microbial insecticides. Compared with chemical synthetic pesticides, microbial pesticides have many advantages \cite{Robert,Sarwar,Starnes}, including: (1) they are harmless to organisms other than the target; (2) pests are not easy to develop resistance; (3) they can protect natural enemies of pests; (4) they do not pollute the environment. These characteristics make microbial pesticides a class of pesticides suitable for integrated pest control.


Entomopathogenic nematode is a kind of new-type and promising microbial pesticide \cite{Taha,Corne,Khoury}. It releases a kind of symbiotic bacteria in its intestine into the blood cavity of the host insect, and then the symbiotic bacteria multiply in the blood cavity and produce antibacterial substances and toxins, causing the host insect to suffer from sepsis and die. Guangjun Ren, deputy dean and researcher of Sichuan Academy of Agricultural Sciences, said, ``Entomopathogenic nematode, as specialized parasitic natural enemies of insects, is a kind of microbial pesticide with the dual characteristics of natural enemies and pathogenic microorganisms, and is an important biological control factor for pests. It can effectively control pests, and it is safe for non-target organisms and the environment. Therefore, it has great application potential in the sustainable management of pests".

\subsection{Model formulation}

In 2009, Wang and Chen \cite{Wang2009} formulated the following mathematical model in order to investigate the dynamics of nematodes attacking pests:
\begin{equation}\label{eq1}
\left\{\begin{aligned}
&\frac{dx}{dt}=rx-cxy,\\
&\frac{dy}{dt}=cxy^2-my,
\end{aligned}\right.
\end{equation}
where $x(t),y(t)$ denotes the density of pests and entomopathogenic nematodes, respectively. $r$ denotes the birth rate of pests and $m$ denotes the death rate of entomopathogenic nematodes. Moreover, the effects of nematodes' predation behavior on pests and nematodes are expressed as $-cxy$ and $+cxy^2$. Subsequently, Wang and Chen \cite{Wang2011} used the Poincar\'e map to analyze dynamic behaviors of the impulsive state of model \eqref{eq1}. In 2017, Wang et al. \cite{Wang2017} considered system \eqref{eq1} with the Monod growth rate. And in 2021, Wang \cite{Wang2021} studied model \eqref{eq1} with density dependent for pests.



{ Many studies have shown that some species of microorganism A (or plant extracts) can inhibit other species of microorganism B (bacteria or fungi) \cite{Okeniyi2020,Okeniyi2017,Coetser2005}, and this inhibiting effect can reduce the birth rate of microorganism B, and the strength of its influence is proportional to the inhibiting factor and the density of microorganism A. Consequently, the density of $B$ is inversely proportional to the density of $A$. That is,  the density of $B$ varies inverselyas the density of $A$. } In fact, this inhibiting effect is similar to the fear effect of predators in larger populations \cite{Barman,Gao,Sa,Wang2,Wang1,Wang,H}. Thus, using an inverse proportional function to modelling the inhibiting effect is very reasonable.
In this paper, we add an inhibiting factor to the original model \eqref{eq1} to modify pests birth rate, and establish the following mathematical model to modelling  a microbial pesticide model with inhibiting effect:
\begin{equation}\label{eq1.1}
\left\{\begin{aligned}
&\frac{dx}{dt}=\frac{rx}{1+ky}-cxy,\\
&\frac{dy}{dt}=cxy^2-my,
\end{aligned}\right.
\end{equation}
where $\frac{1}{1+ky}$ is the inhibition function, $k$ is the level of inhibition, it means that if the density of nematodes or the level of inhibition was zero, there was no effect on pests; with the increase of the density of nematodes or the level of inhibition, the birth rate of pests would decrease.

Through the analysis of system \eqref{eq1.1}, we derive that in any case, there will be not a steady state, that is, the density of pests will keep increasing, reach destructive numbers and take a toll on the economy. Therefore, we will through continuous release of nematodes to control the density of pests. And then we have the following model:
\begin{equation}\label{eq1.2}
\left\{\begin{aligned}
&\frac{dx}{dt}=\frac{rx}{1+ky}-cxy,\\
&\frac{dy}{dt}=cxy^2-my+u,
\end{aligned}\right.
\end{equation}
where $u$ is the release rate of entomopathogenic nematodes, other parameters are the same as in systems \eqref{eq1} and \eqref{eq1.1}. And all parameters $r,k,c,m,u$ are positive.

\subsection{Organization of paper}

The structure of this paper is organized as follows. Next section is to study the stability of the equilibria and the nonexistence of limit cycle of system \eqref{eq1.1}. In Section 3, we prove the existence of Hopf bifurcation of system \eqref{eq1.2}. In Section 4, some numerical simulations are presented to show the feasibility of the main results. Finally, a conclusion ends the paper.

\section{\Large{Dynamic behavior of system \eqref{eq1.1}}}

\par\noindent

\begin{Theorem} {\label{Theorem 2.1}} System \eqref{eq1.1} always has two equilibria, a boundary equilibrium $E_0(0,0)$ and a positive equilibrium $E_1(x_1,y_1)$. Furthermore, $E_0$ is a saddle, $E_1$ is an unstable node or focus.
\end{Theorem}

\begin{proof} The equilibria of  \eqref{eq1.1} satisfy the  equations
\begin{equation}\label{eq2.1}
\left\{\begin{aligned}
&\frac{rx}{1+ky}-cxy=0,\\
&cxy^2-my=0.
\end{aligned}\right.
\end{equation}
Obviously, equation \eqref{eq2.1} has nonnegative solutions
$$x_0=0,~~~y_0=0,$$
and
$$x_1=\frac{2km}{-c+\sqrt{c^2+4crk}},~~~y_1=\frac{-c+\sqrt{c^2+4crk}}{2ck}.$$
Consider the Jacobian matrix of system \eqref{eq1.1}
\begin{equation*}
J=\left(\begin{array}{cc}
\frac{r}{1+ky}-cy & -\frac{rkx}{(1+ky)^2}-cx\\
cy^2 & 2cxy-m\\
\end{array}
\right).
\end{equation*}
The Jacobian matrix at $E_0(0,0)$ is
\begin{equation*}
J(E_0)=\left(\begin{array}{cc}
r & 0\\
0 & -m\\
\end{array}
\right).
\end{equation*}
It is easy to see that $$\lambda_1(E_0)=r>0,~~~\lambda_2(E_0)=-m<0,$$
thus, $E_0$ is a saddle.
The Jacobian matrix at $E_1(x_1,y_1)$ is
\begin{equation*}
J(E_1)=\left(\begin{array}{cc}
0 & -\frac{rkx_1}{(1+ky_1)^2}-cx_1\\
cy_1^2 & m\\
\end{array}
\right).
\end{equation*}
According to the relationship between the matrix and its corresponding eigenvalues, it can be known
$$\lambda_1(E_1)+\lambda_2(E_1)=tr(J(E_1))=m>0,$$
$$\lambda_1(E_1)\cdot\lambda_2(E_1)=det J(E_1)=cy_1^2\left(\frac{rkx_1}{(1+ky_1)^2}cx_1\right)>0,$$
and then, we have $$Re(\lambda_1(E_1))>0,~~~Re(\lambda_2(E_1))>0.$$
Thus, $E_1$ is an unstable node or focus.
\end{proof}

\begin{Theorem} {\label{Theorem 2.2}} System \eqref{eq1.1} has no limit cycle in the first quadrant.
\end{Theorem}

\begin{proof}
Set $$P(x,y)=\frac{rx}{1+ky}-cxy,~~~Q(x,y)=cxy^2-my.$$
Choosing Dulac functon $$B(x,y)=\frac{1}{xy}.$$
Functions $P, Q, B$ are continuously differentiable in the first quadrant, and
\begin{equation*}
\begin{split}
\frac{\partial(BP)}{\partial x}+\frac{\partial(BQ)}{\partial y}
=\frac{\partial}{\partial x}\left[\frac{1}{y(1+ky)}-c\right]+\frac{\partial}{\partial y}\left(cy-\frac{m}{x}\right)=c>0.
\end{split}
\end{equation*}
It follows from Bendixon-Dulac Theorem that there is no limit cycle in the first quadrant.
\end{proof}


From the above analysis, it can be seen that system \eqref{eq1.1} will not approach a stable state at any time and under any conditions. This is really bad from a biological standpoint. The increasing number of pests will cause damage to the environment and the economy. Therefore, it is necessary to take necessary measures to improve this situation. Below we discuss dynamic behavior of model \eqref{eq1.2} under the condition of continuous release of nematodes.

\section{\Large{Dynamic behavior of system \eqref{eq1.2}}}

\par\noindent

At first, we nondimensionalise system \eqref{eq1.2} by writing
$$\bar{y}=\frac{c}{r}y,~~~\tau=rt,$$
then system \eqref{eq1.2} can be turned into
\begin{equation}\label{eq3.1}
\left\{\begin{aligned}
&\frac{dx}{d\tau}=\frac{x}{1+\frac{kr}{c}\bar{y}}-x\bar{y},\\
&\frac{d\bar{y}}{d\tau}=x\bar{y}^2-\frac{m}{r}\bar{y}+\frac{cu}{r^2}.
\end{aligned}\right.
\end{equation}
Taking $$\bar{k}=\frac{kr}{c},~~~\bar{m}=\frac{m}{r},~~~\bar{u}=\frac{cu}{r^2},$$
and still replace $\bar{y},\tau,\bar{k},\bar{m},\bar{u}$ with the original variable $y,t,k,m,u$, and then, system \eqref{eq3.1} becomes
\begin{equation}\label{eq3.2}
\left\{\begin{aligned}
&\frac{dx}{dt}=\frac{x}{1+ky}-xy,\\
&\frac{dy}{dt}=xy^2-my+u.
\end{aligned}\right.
\end{equation}

\begin{Theorem} {\label{Theorem 3.1}} System \eqref{eq3.2} always has a boundary equilibrium $E_2(0,y_2)$. In addition,

$(i)$ if $u<u_0$, then system \eqref{eq3.2} has a positive equilibrium $E_3(x_3,y_3)$;

$(ii)$ if $u\geq u_0$, then system \eqref{eq3.2} has no other equilibrium except $E_2$.
\end{Theorem}

\begin{proof}
The equilibria of system \eqref{eq3.2} satisfy
\begin{equation}\label{eq3.2a}
\left\{\begin{aligned}
&\frac{x}{1+ky}-xy=0,\\
&xy^2-my+u=0.
\end{aligned}\right.
\end{equation}
By calculations, the above system has a fixed solution
$$x_2=0,~~y_2=\frac{u}{m}.$$
In addition, from $\frac{1}{1+ky}-y=0$, we have
$$y_3=\frac{\sqrt{1+4k}-1}{2k},$$and then, it follows from the second equation of system \eqref{eq3.2a} that
$$x_3=\frac{my_3-u}{y_3^2}.$$
It can be seen from non-negativity of the equilibria that
$x_3>0$, i.e., $u<my_3:=u_0$.
\end{proof}

\begin{Theorem} {\label{Theorem 3.2}} For equilibrium $E_2(0,y_2)$,

$(i)$ if $u<u_0$, then $E_2$ is a saddle;

$(ii)$ if $u>u_0$, then $E_2$ is a stable node;

$(iii)$ if $u=u_0$, then $E_2$ is an attracting saddle node.
\end{Theorem}

\begin{proof}
The Jacobian matrix of system \eqref{eq3.2} is
\begin{equation*}
J=\left(\begin{array}{cc}
\frac{1}{1+ky}-y & -\frac{kx}{(1+ky)^2}-x\\
y^2 & 2xy-m\\
\end{array}
\right).
\end{equation*}
Thus, at $E_2(0,y_2)$,
\begin{equation*}
J(E_2)=\left(\begin{array}{cc}
\frac{1}{1+ky_2}-y_2 & 0\\
y_2^2 & -m\\
\end{array}
\right).
\end{equation*}
The eigenvalues of $J(E_2)$ are $$\lambda_1(J(E_2))=\frac{1}{1+ky_2}-y_2,~~~\lambda_2(J(E_2))=-m<0.$$

$(i)$ If $y_2<y_3$, i.e., $u<u_0$, $\lambda_1(J(E_2))=\frac{1}{1+ky_2}-y_2>0$,  $E_2$ is a saddle;

$(ii)$ if $y_2>y_3$, i.e., $u>u_0$, $\lambda_1(J(E_2))=\frac{1}{1+ky_2}-y_2<0$, $E_2$ is a stable node;

$(iii)$ if $y_2=y_3$, i.e., $u=u_0$, in where $E_2$ and $E_3$ coincide as a point, at this time, $$\lambda_1(J(E_2))=\frac{1}{1+ky_2}-y_2=0.$$ In order to recognize the type and stability of $E_2$, at first, translating $E_2$ to the origin by transformation $(X,Y)=(x,y-y_2)$, and performing Taylor expansion of system \eqref{eq3.2} at the origin to the third order, and noticing that
$$\frac{1}{1+ky_2}-y_2=0,~~~my_2-u=0.$$ Thus, we have
\begin{equation}\label{eq3.3}
\left\{\begin{aligned}
&\frac{dX}{dt}=-(ky_2^2+1)XY+k^2y_2^3XY^2+o(|X,Y|^4),\\
&\frac{dY}{dt}=y_2^2X-mY+2y_2XY+XY^2.
\end{aligned}\right.
\end{equation}
Taking transformation $(\tilde{X},\tilde{Y})=(X,X-\frac{m}{y_2^2}Y)$, system \eqref{eq3.3} gives
\begin{equation}\label{eq3.4}
\left\{\begin{aligned}
\frac{d\tilde{X}}{dt}=&-\frac{y_2^2(ky_2^2+1)}{m}\tilde{X}^2+\frac{y_2^2(ky_2^2+1)}{m}\tilde{X}\tilde{Y}+\frac{k^2y_2^7}{m^2}\tilde{X}^3-\frac{2k^2y_2^7}{m^2}\tilde{X}^2\tilde{Y}+\frac{k^2y_2^7}{m^2}\tilde{X}\tilde{Y}^2+o(|X,Y|^4),\\
\frac{d\tilde{Y}}{dt}=&-m\tilde{Y}-\left(\frac{y_2^2(ky_2^2+1)}{m}+2y_2\right)\tilde{X}^2+\left(\frac{y_2^2(ky_2^2+1)}{m}+2y_2\right)\tilde{X}\tilde{Y}+\left(\frac{k^2y_2^7}{m^2}-\frac{y_2^2}{m}\right)\tilde{X}^3\\
&-\left(\frac{2k^2y_2^7}{m^2}-\frac{2y_2^2}{m}\right)\tilde{X}^2\tilde{Y}+\left(\frac{k^2y_2^7}{m^2}-\frac{y_2^2}{m}\right)\tilde{X}\tilde{Y}^2+o(|X,Y|^4).
\end{aligned}\right.
\end{equation}
Now we apply time rescaling $\tau=-mt$, and system \eqref{eq3.4} transformed into the standard form
\begin{equation}\label{eq3.4}
\left\{\begin{aligned}
\frac{d\tilde{X}}{d\tau}=&\frac{y_2^2(ky_2^2+1)}{m^2}\tilde{X}^2-\frac{y_2^2(ky_2^2+1)}{m^2}\tilde{X}\tilde{Y}-\frac{k^2y_2^7}{m^3}\tilde{X}^3+\frac{2k^2y_2^7}{m^3}\tilde{X}^2\tilde{Y}-\frac{k^2y_2^7}{m^3}\tilde{X}\tilde{Y}^2+o(|X,Y|^4),\\
\frac{d\tilde{Y}}{d\tau}=&\tilde{Y}+\left(\frac{y_2^2(ky_2^2+1)}{m^2}+\frac{2y_2}{m}\right)\tilde{X}^2-\left(\frac{y_2^2(ky_2^2+1)}{m^2}+\frac{2y_2}{m}\right)\tilde{X}\tilde{Y}-\left(\frac{k^2y_2^7}{m^3}-\frac{y_2^2}{m^2}\right)\tilde{X}^3\\
&+\left(\frac{2k^2y_2^7}{m^3}-\frac{2y_2^2}{m^2}\right)\tilde{X}^2\tilde{Y}-\left(\frac{k^2y_2^7}{m^3}-\frac{y_2^2}{m^2}\right)\tilde{X}\tilde{Y}^2+o(|X,Y|^4).
\end{aligned}\right.
\end{equation}
From $\frac{d\tilde{Y}}{d\tau}=0$, we have implicit function $\tilde{Y}=\phi(\tilde{X})=0$, then
\begin{equation*}
\frac{d\tilde{X}}{d\tau}=\frac{y_2^2(ky_2^2+1)}{m^2}\tilde{X}^2-\frac{k^2y_2^7}{m^3}\tilde{X}^3+o(|X|^4).
\end{equation*}
In view of $\frac{y_2^2(ky_2^2+1)}{m^2}>0$, $E_2$ is an attracting saddle node, which can be obtained from \cite{Zhang}[Theorem 7.1], and this theorem is a powerful tool to study the bifurcation of planar system which has been applied to many models in real world \cite{ChenLJ,Guan,HuangJC1,HuangJC2,Song1,Song2,Wei,Wei1}.
\end{proof}

\begin{Theorem} {\label{Theorem 3.3}} For equilibrium $E_3(x_3,y_3)$,

$(i)$ if $u<\frac{u_0}{2}$, then $E_3$ is an unstable focus;

$(ii)$ if $\frac{u_0}{2}<u<u_0$, then $E_3$ is a stable focus;

$(iii)$ if $u=\frac{u_0}{2}$, then $E_3$ is a center type stable focus, and at the moment, system \eqref{eq3.2} undergoes a Hopf bifurcation.
\end{Theorem}
\begin{proof}
The Jacobian matrix of system \eqref{eq3.2} in $E_3$ is
\begin{equation}\label{eq3.4a}
J(E_3)=\left(\begin{array}{cc}
0 & -kx_3y_3^2-x_3\\
y_3^2 & 2x_3y_3-m\\
\end{array}
\right).
\end{equation}
According to the relationship between the matrix and its corresponding eigenvalues, we have
$$\lambda_1(E_3)+\lambda_2(E_3)=tr(J(E_3))=2x_3y_3-m,$$
$$\lambda_1(E_3)\cdot\lambda_2(E_3)=det J(E_3)=y_3^2(kx_3y_3^2+x_3)>0.$$

We first prove $(iii)$, $u=\frac{u_0}{2}$, i.e., $y_3=\frac{2u}{m}$, in this state, $$\lambda_1(E_3)+\lambda_2(E_3)=0.$$
Translating $E_3$ to the origin by transformation $(X,Y)=(x-x_3,y-y_3)$, and performing Taylor expansion of system \eqref{eq3.2} at the origin to the third order, and noticing that
$2x_3y_3=m$. Thus, we have
\begin{equation}\label{eq3.5}
\left\{\begin{aligned}
\frac{dX}{dt}=&-x_3(ky_3^2+1)Y-(ky_3^2+1)XY+k^2x_3y_3^3Y^2+k^2y_3^3XY^2-k^3x_3y_3^4Y^3+o(|X,Y|^4),\\
\frac{dY}{dt}=&y_3^2X+2y_3XY+x_3Y^2+XY^2.
\end{aligned}\right.
\end{equation}
Noting $\omega:=\sqrt{x_3(ky_3^2+1)}$, and taking transformation $(\bar{X},\bar{Y})=\left(\frac{y_3}{\omega}X,Y\right)$, system \eqref{eq3.5} gives
\begin{equation}\label{eq3.6}
\left\{\begin{aligned}
\frac{d\bar{X}}{dt}=&-y_3\omega\bar{Y}-(ky_3^2+1)\bar{X}\bar{Y}+\frac{k^2}{\omega}x_3y_3^4\bar{Y}^2+k^2y_3^3\bar{X}\bar{Y}^2-\frac{k^3}{\omega}x_3y_3^5\bar{Y}^3+o(|\bar{X},\bar{Y}|^4),\\
\frac{d\bar{Y}}{dt}=&y_3\omega\bar{X}+2\omega\bar{X}\bar{Y}+x_3\bar{Y}^2+\frac{\omega}{y_3}\bar{X}\bar{Y}^2.
\end{aligned}\right.
\end{equation}
Replacing the coefficients of $\bar{X}^i\bar{Y}^j$ $(i,j=0,1,2,3)$ in $\frac{d\bar{X}}{dt}$ and $\frac{d\bar{Y}}{dt}$ with $A_{ij}$ and $B_{ij}$ respectively, then system \eqref{eq3.6} becomes
\begin{equation}\label{eq3.7}
\left\{\begin{aligned}
&\frac{d\bar{X}}{dt}=-y_3\omega\bar{Y}+A_{11}\bar{X}\bar{Y}+A_{02}\bar{Y}^2+A_{12}\bar{X}\bar{Y}^2+A_{03}\bar{Y}^3+o(|\bar{X},\bar{Y}|^4),\\
&\frac{d\bar{Y}}{dt}=y_3\omega\bar{X}+B_{11}\bar{X}\bar{Y}+B_{02}\bar{Y}^2+B_{12}\bar{X}\bar{Y}^2.
\end{aligned}\right.
\end{equation}
According to the calculation method of the third focus value, we obtain the third focus value of system \eqref{eq3.7} at the origin
\begin{equation*}
\begin{split}
&\frac{\pi}{4y_3\omega}A_{12}-\frac{\pi}{4(y_3\omega)^2}(2A_{02}B_{02}-A_{11}A_{02}+B_{11}B_{02})\\
=&\frac{\pi}{4y_3\omega}\times k^2y_3^3-\frac{\pi}{4(y_3\omega)^2}\left[\frac{2k^2}{\omega}x_3^2y_3^4+\frac{k^2}{\omega}x_3y_3^4(ky_3^2+1)+2\omega x_3\right]\\
=&-\frac{\pi}{2\omega}\left(\frac{k^2}{\omega^2}x_3^2y_3^2+\frac{x_3}{y_3^2}\right)<0,
\end{split}
\end{equation*}
which implies $E_3$ is a center type stable focus \cite{J}[Chapters 2.3 and 7.1]. In this state, the eigenvalues of its Jacobian matrix $J(E_3)$ are a pair of conjugate pure virtual eigenvalues $\lambda_{1,2}=\pm iy_3\omega$. When $u$ changes near $\frac{u_0}{2}$, $J(E_3)$ has a pair of conjugate complex eigenvalues $\lambda_{1,2}=\alpha(u)\pm i\beta(u)$, where
$$\alpha(u)=\frac{1}{2}tr(J(E_3))=\frac{1}{2}(2x_3y_3-m)=\frac{m}{2}-\frac{u}{y_3},~~~
\beta(u)=\sqrt{det J(E_3)-\alpha^2(u)}.$$
Since $\alpha'(u)\mid_{u=\frac{u_0}{2}}=-\frac{1}{y_3}<0$, the transversality condition holds, it follows from Poincar\'e-Andronov-Hopf bifurcation theory \cite{Wiggins}[Theorem 3.1.3] that system \eqref{eq3.2} undergoes a Hopf bifurcation in $E_3$ when $u=\frac{u_0}{2}$;

$(i)$ if $y_3>\frac{2u}{m}$, i.e., $u<\frac{u_0}{2}$, $\lambda_1(E_3)+\lambda_2(E_3)>0$, $J(E_3)$ has a pair of conjugate complex eigenvalues, and
the real part is greater than $0$, then $E_3$ is an unstable focus;

$(ii)$ if $\frac{u}{m}<y_3<\frac{2u}{m}$, i.e., $\frac{u_0}{2}<u<u_0$, $\lambda_1(E_3)+\lambda_2(E_3)<0$, $J(E_3)$ has a pair of conjugate complex eigenvalues, and
the real part is less than $0$, then $E_3$ is a stable focus.
\end{proof}

\begin{center}
 \makeatletter\def\@captype{table}\makeatother
 \newcommand{\tabincell}[2]{\begin{tabular}{@{}#1@{}}#2\end{tabular}}
 \caption{ Equilibria and their stability in system \eqref{eq3.2}}
 \small
   \label{TableAA}
  \centering
  \renewcommand\arraystretch{1.3}
  \setlength{\tabcolsep}{1mm}{
  \begin{tabular}{lll}
\hline
Equilibrium        & \hspace{4em}Existence &  \hspace{4em}Type \\
\cline{1-3}
 $E_2(0,y_2)$      & \hspace{4em}Always exists &  \hspace{4em}\tabincell{l}{
 $u<u_0$, saddle \\$u>u_0$, stable node \\$u=u_0$, attracting saddle node } \\
 \cline{1-3}
 $E_3(x_3,y_3)$      & \hspace{4em}$u<u_0$  &  \hspace{4em}\tabincell{l}{$u<\frac{u_0}{2}$, unstable focus  \\$\frac{u_0}{2}<u<u_0$, stable focus \\ $u=\frac{u_0}{2}$, center type stable focus} \\
\hline
\end{tabular}}
\end{center}

\begin{Theorem} {\label{Theorem 3.4}} System \eqref{eq3.2} undergoes a Hopf bifurcation at $E_3$ when $u=\frac{u_0}{2}$, furthermore, the Hopf bifurcation is subcritical, and bifurcation periodic solution is orbitally asymptotic stable.
\end{Theorem}

\begin{proof}
Translating $E_3$ to the origin by transformation $(X,Y)=(x-x_3,y-y_3)$, and performing Taylor expansion of system \eqref{eq3.2} at the origin to the third order,
\begin{equation*}
\left\{\begin{aligned}
\frac{dX}{dt}=&-x_3(ky_3^2+1)Y-(ky_3^2+1)XY+k^2x_3y_3^3Y^2+k^2y_3^3XY^2-k^3x_3y_3^4Y^3+o(|X,Y|^4),\\
\frac{dY}{dt}=&y_3^2X+(2x_3y_3-m)Y+2y_3XY+x_3Y^2+XY^2.
\end{aligned}\right.
\end{equation*}
Rewriting the above system as
\begin{equation}\label{eq3.8}
\left(\begin{array}{c}
\frac{dX}{dt}\\
\frac{dY}{dt}\\
\end{array}\right)=J(E_3)\left(\begin{array}{c}
X\\
Y\\
\end{array}\right)+\left(\begin{array}{c}
f(x,y,u)\\
g(x,y,u)\\
\end{array}\right),
\end{equation}
where $J(E_3)$ is as in \eqref{eq3.4a}, and
\begin{equation*}
\begin{split}
f(x,y,u)&=-(ky_3^2+1)XY+k^2x_3y_3^3Y^2+k^2y_3^3XY^2-k^3x_3y_3^4Y^3+o(|X,Y|^4), \\ g(x,y,u)&=2y_3XY+x_3Y^2+XY^2.
\end{split}
\end{equation*}
Define a matrix $P=\begin{pmatrix}
  1 & 0 \\
  N & M \\
\end{pmatrix}$, where $N=-\frac{\alpha(u)}{x_3(ky_3^2+1)}$, $M=\frac{\beta(u)}{x_3(ky_3^2+1)}$. When $u=\frac{u_0}{2}$, $M=\frac{y_3}{\sqrt{x_3(ky_3^2+1)}}>0$. Then when $u$ changes near $\frac{u_0}{2}$, $P$ is invertible, and $P^{-1}=\begin{pmatrix}
  1 & 0 \\
  -\frac{N}{M} & \frac{1}{M} \\
\end{pmatrix}$, in addition,
\begin{equation*}
P^{-1}J(E_3)P=\begin{pmatrix}
  \alpha(u) & -\beta(u) \\
  \beta(u) & \alpha(u) \\
\end{pmatrix}.
\end{equation*}
By transformation $(X,Y)^T=P(\xi,\eta)^T$, system \eqref{eq3.8} becomes
\begin{equation}\label{eq3.9}
\left(\begin{array}{c}
\frac{d\xi}{dt}\\
\frac{d\eta}{dt}\\
\end{array}\right)=\begin{pmatrix}
  \alpha(u) & -\beta(u) \\
  \beta(u) & \alpha(u) \\
\end{pmatrix}\left(\begin{array}{c}
\xi\\
\eta\\
\end{array}\right)+\left(\begin{array}{c}
F(\xi,\eta,u)\\
G(\xi,\eta,u)\\
\end{array}\right),
\end{equation}
where
\begin{equation*}
\begin{split}
F(\xi,\eta,u)=&f(\xi,\eta,u)\\
=&-(ky_3^2+1)\xi(N\xi+M\eta)+k^2x_3y_3^3(N\xi+M\eta)^2+k^2y_3^3\xi(N\xi+M\eta)^2\\
&-k^3x_3y_3^4(N\xi+M\eta)^3+o(|\xi,\eta|^4),\\ G(\xi,\eta,u)=&-\frac{N}{M}f(\xi,\eta,u)+\frac{1}{M}g(\xi,\eta,u)\\
=&-\frac{N}{M}[-(ky_3^2+1)\xi(N\xi+M\eta)+k^2x_3y_3^3(N\xi+M\eta)^2+k^2y_3^3\xi(N\xi+M\eta)^2\\
&-k^3x_3y_3^4(N\xi+M\eta)^3+o(|\xi,\eta|^4)]+\frac{1}{M}[2y_3\xi(N\xi+M\eta)+x_3(N\xi+M\eta)^2\\
&+\xi(N\xi+M\eta)^2].
\end{split}
\end{equation*}
We can write system \eqref{eq3.9} into the following polar form
\begin{equation*}
\left\{\begin{aligned}
&\dot{r}=\alpha(u)r+\alpha_1(u)r^3+\cdots,\\
&\dot{\theta}=\beta(u)+\beta_1(u)r^2+\cdots,
\end{aligned}\right.
\end{equation*}
performing Taylor expansion of the above system at $u=\frac{u_0}{2}$, we have
\begin{equation*}
\left\{\begin{aligned}
\dot{r}=&\alpha'\left(\frac{u_0}{2}\right)\left(u-\frac{u_0}{2}\right)r+\alpha_1\left(\frac{u_0}{2}\right)r^3+o\left(\left(u-\frac{u_0}{2}\right)^2r,\left(u-\frac{u_0}{2}\right)r^3,r^5\right),\\
\dot{\theta}=&\beta\left(\frac{u_0}{2}\right)+\beta'\left(\frac{u_0}{2}\right)\left(u-\frac{u_0}{2}\right)+\beta_1\left(\frac{u_0}{2}\right)r^2+o\left(\left(u-\frac{u_0}{2}\right)^2,\left(u-\frac{u_0}{2}\right)r^2,r^4\right).
\end{aligned}\right.
\end{equation*}
In order to examine the direction of the Hopf bifurcation and the stability of Hopf bifurcation periodic solution, we have to determine the sign of $\alpha_1(\frac{u_0}{2})$, where
\begin{equation*}
\begin{split}
\alpha_1\left(\frac{u_0}{2}\right)=&\frac{1}{16}(F_{\xi\xi\xi}+F_{\xi\eta\eta}+G_{\xi\xi\eta}+G_{\eta\eta\eta})+\frac{1}{16\beta\left(\frac{u_0}{2}\right)}[F_{\xi\eta}(F_{\xi\xi}+F_{\eta\eta})-G_{\xi\eta}(G_{\xi\xi}+G_{\eta\eta})\\
&-F_{\xi\xi}G_{\xi\xi}+F_{\eta\eta}G_{\eta\eta}].
\end{split}
\end{equation*}
All partial derivatives in the above formula are calculated at $(\xi,\eta,u)=\left(0,0,\frac{u_0}{2}\right)$.
By calculating, $$F_{\xi\xi\xi}=F_{\xi\xi}=G_{\xi\xi\eta}=G_{\eta\eta\eta}=G_{\xi\xi}=0,$$
$$F_{\xi\eta\eta}=2k^2y_3^3M^2\left(\frac{u_0}{2}\right),~~~F_{\xi\eta}=-(ky_3^2+1)M\left(\frac{u_0}{2}\right),$$
$$F_{\eta\eta}=2k^2x_3y_3^3M^2\left(\frac{u_0}{2}\right),~~~G_{\xi\eta}=2y_3,~~~G_{\eta\eta}=2x_3M\left(\frac{u_0}{2}\right).$$
Noticing that $$M\left(\frac{u_0}{2}\right)=\frac{\beta\left(\frac{u_0}{2}\right)}{x_3(ky_3^2+1)},~~~\beta\left(\frac{u_0}{2}\right)=y_3\sqrt{x_3(ky_3^2+1)},~~~y_3=\frac{u_0}{m}=\frac{\sqrt{1+4k}-1}{2k}.$$ Therefore, we have
\begin{equation*}
\begin{split}
\alpha_1\left(\frac{u_0}{2}\right)=&\frac{1}{16}F_{\xi\eta\eta}+\frac{1}{16\beta\left(\frac{u_0}{2}\right)}(F_{\xi\eta}F_{\eta\eta}-G_{\xi\eta}G_{\eta\eta}+F_{\eta\eta}G_{\eta\eta})\\
=&\frac{1}{16}\left[2k^2y_3^3M^2\left(\frac{u_0}{2}\right)+\frac{1}{\beta\left(\frac{u_0}{2}\right)}\left(-2k^2x_3y_3^3(ky_3^2+1)M^3\left(\frac{u_0}{2}\right)-4x_3y_3M\left(\frac{u_0}{2}\right)\right.\right.\\
&\left.\left.+4k^2x_3^2y_3^3M^3\left(\frac{u_0}{2}\right)\right)\right]\\
=&\frac{x_3y_3M\left(\frac{u_0}{2}\right)}{4\beta\left(\frac{u_0}{2}\right)}\left(\frac{k^2y_3^4}{ky_3^2+1}-1\right)\\
=&\frac{x_3y_3M\left(\frac{u_0}{2}\right)}{4\beta\left(\frac{u_0}{2}\right)}\left(\frac{(\sqrt{1+4k}-1)^4}{4k((\sqrt{1+4k}-1)^2)+4k}-1\right),
\end{split}
\end{equation*}
taking variable substitution $\sqrt{1+4k}-1:=\kappa$, then
\begin{equation*}
\begin{split}
\alpha_1\left(\frac{u_0}{2}\right)=&\frac{x_3y_3M\left(\frac{u_0}{2}\right)}{4\beta\left(\frac{u_0}{2}\right)}\left(\frac{\kappa^4}{((\kappa+1)^2-1)(\kappa^2+(\kappa+1)^2-1)}-1\right)\\
=&\frac{x_3y_3M\left(\frac{u_0}{2}\right)}{4\beta\left(\frac{u_0}{2}\right)}\left(\frac{\kappa^2}{(\kappa+2)(2\kappa+2)}-1\right)<0.
\end{split}
\end{equation*}
The first Lyapunov coefficient
$$l_1\left(\frac{u_0}{2}\right)=-\frac{\alpha_1\left(\frac{u_0}{2}\right)}{\alpha'\left(\frac{u_0}{2}\right)}<0.$$
Therefore, the direction of the Hopf bifurcation is subcritical \cite{Yuri}[Chapter 3.4], and bifurcation periodic solution is orbitally asymptotic stable.
\end{proof}

\section{\Large{Examples and their numerical simulations }}

\par\noindent

In this section, we give an example and figures to illustrate our results.

\begin{Example}
Consider the following system
\begin{equation}\label{eq4.1}
\left\{\begin{aligned}
&\frac{dx}{dt}=\frac{2x}{1+0.5y}-2xy,\\
&\frac{dy}{dt}=2xy^2-0.4y+u.
\end{aligned}\right.
\end{equation}
Comparing system \eqref{eq4.1} with system \eqref{eq1.2}, we see that $r=2$, $k=0.5$, $c=2$, $m=0.4$. Furthermore, $\bar{y}=\frac{c}{r}y=y$, $\tau=2t$, $\bar{k}=\frac{kr}{c}=0.5$, $\bar{m}=\frac{m}{r}=0.2$, $\bar{u}=\frac{cu}{r^2}=\frac{u}{2},$
and system \eqref{eq4.1} becomes
\begin{equation}\label{eq4.2}
\left\{\begin{aligned}
&\frac{dx}{d\tau}=\frac{x}{1+0.5\bar{y}}-x\bar{y},\\
&\frac{d\bar{y}}{d\tau}=x\bar{y}^2-0.2\bar{y}+\bar{u}.
\end{aligned}\right.
\end{equation}
For the convenience of marking on the graphs, we still replace $\bar{y},\tau,\bar{k},\bar{m},\bar{u}$ with the original variable $y,t,k,m,u$ until we make a prompt below. And then, model \eqref{eq4.2} can be rewritten as
\begin{equation}\label{eq4.3}
\left\{\begin{aligned}
&\frac{dx}{dt}=\frac{x}{1+0.5y}-xy,\\
&\frac{dy}{dt}=xy^2-0.2y+u.
\end{aligned}\right.
\end{equation}
By calculating, we obtain
$y_2=\frac{u}{m}=5u$, $y_3=\frac{\sqrt{1+4k}-1}{2k}=\sqrt{3}-1$, and
$u_0=my_3=0.2(\sqrt{3}-1)$. In the following, we calculate and numerically simulate the dynamic behavior of system \eqref{eq4.3} by taking different values for $u$.

$(1)$ Taking $u=0.1(\sqrt{2}-1)$, $u<\frac{u_0}{2}$, system \eqref{eq4.3} has two equilibria,
$E_2(0,y_2)=(0,5u)\approx(0,0.207)$ is a saddle, and $E_3=\left(\frac{0.2y_3-u}{y_3^2},y_3\right)\approx(0.196,0.732)$ is an unstable focus. See Fig.1 (a);

$(2)$ taking $u=0.1$, $\frac{u_0}{2}<u<u_0$, system \eqref{eq4.3} has two equilibria,
$E_2(0,y_2)=(0,5u)=(0,0.5)$ is a saddle, and $E_3=\left(\frac{0.2y_3-u}{y_3^2},y_3\right)\approx(0.087,0.732)$ is a stable focus. See Fig.1 (b);

$(3)$ taking $u=0.2(\sqrt{3}-1)$, $u=u_0$, system \eqref{eq4.3} has a equilibrium,
$E_2(0,y_2)=(0,5u)\approx(0,0.732)$ is an attracting saddle node. See Fig.1 (c);

$(4)$ taking $u=0.2$, $u>u_0$, system \eqref{eq4.3} has a equilibrium,
$E_2(0,y_2)=(0,5u)=(0,1)$ is a stable node. See Fig.1 (d);

\begin{figure}[H]
\centering
\begin{minipage}[c]{0.5\textwidth}
\centering
\includegraphics[height=5cm,width=6cm]{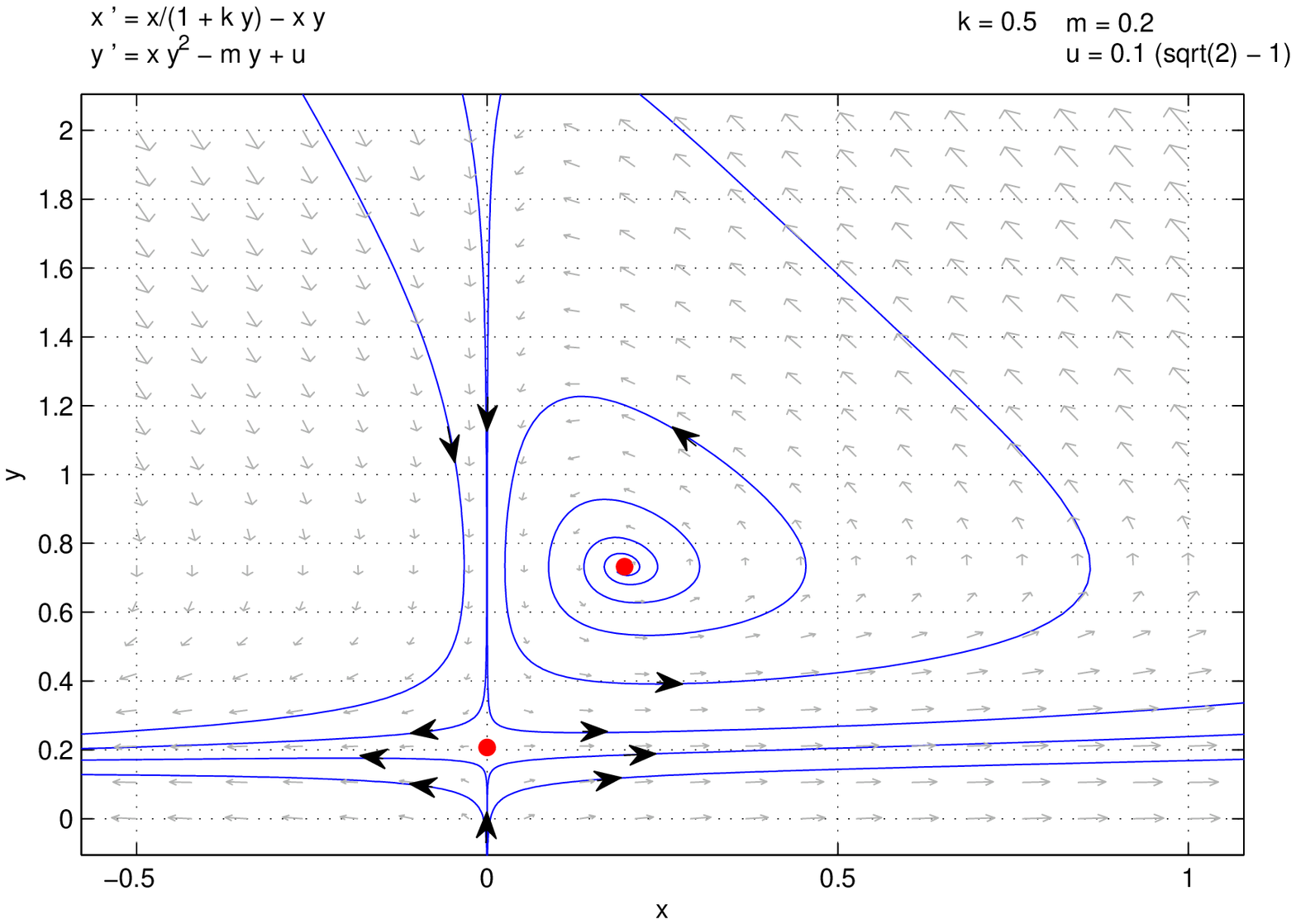}
\end{minipage}%
\begin{minipage}[c]{0.5\textwidth}
\centering
\includegraphics[height=5cm,width=6cm]{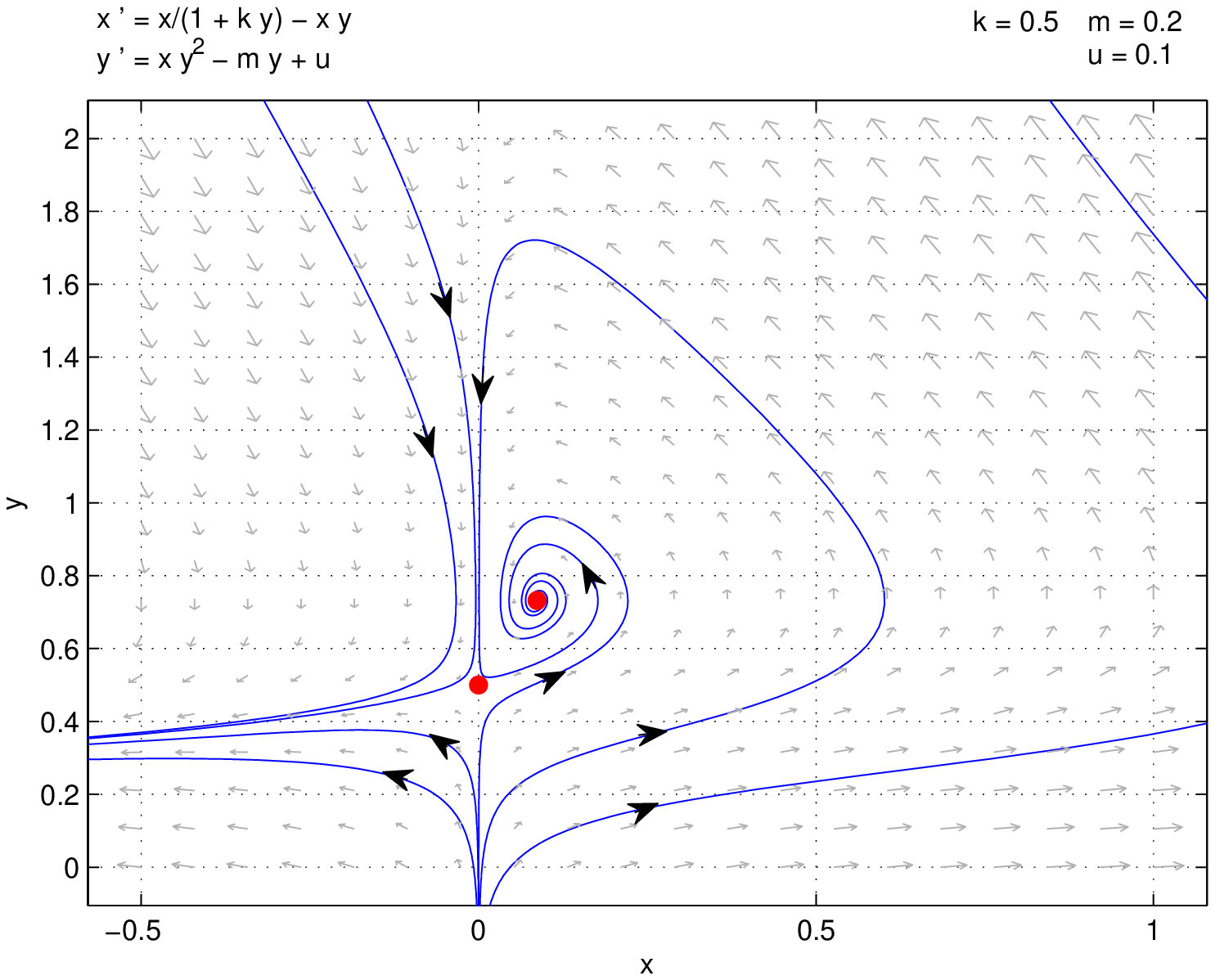}
\end{minipage}
(a) $u<\frac{u_0}{2}$
\ \ \  \ \ \ \ \ \ \ \ \ \
\ \ \  \ \ \ \ \ \ \ \ \ \
\ \ \  \ \ \ \ \ \ \ \ \ \
\ \ \  \ \ \ \ \ \ \ (b) $\frac{u_0}{2}<u<u_0$
\end{figure}

\begin{figure}[H]
\centering
\begin{minipage}[c]{0.5\textwidth}
\centering
\includegraphics[height=5cm,width=6cm]{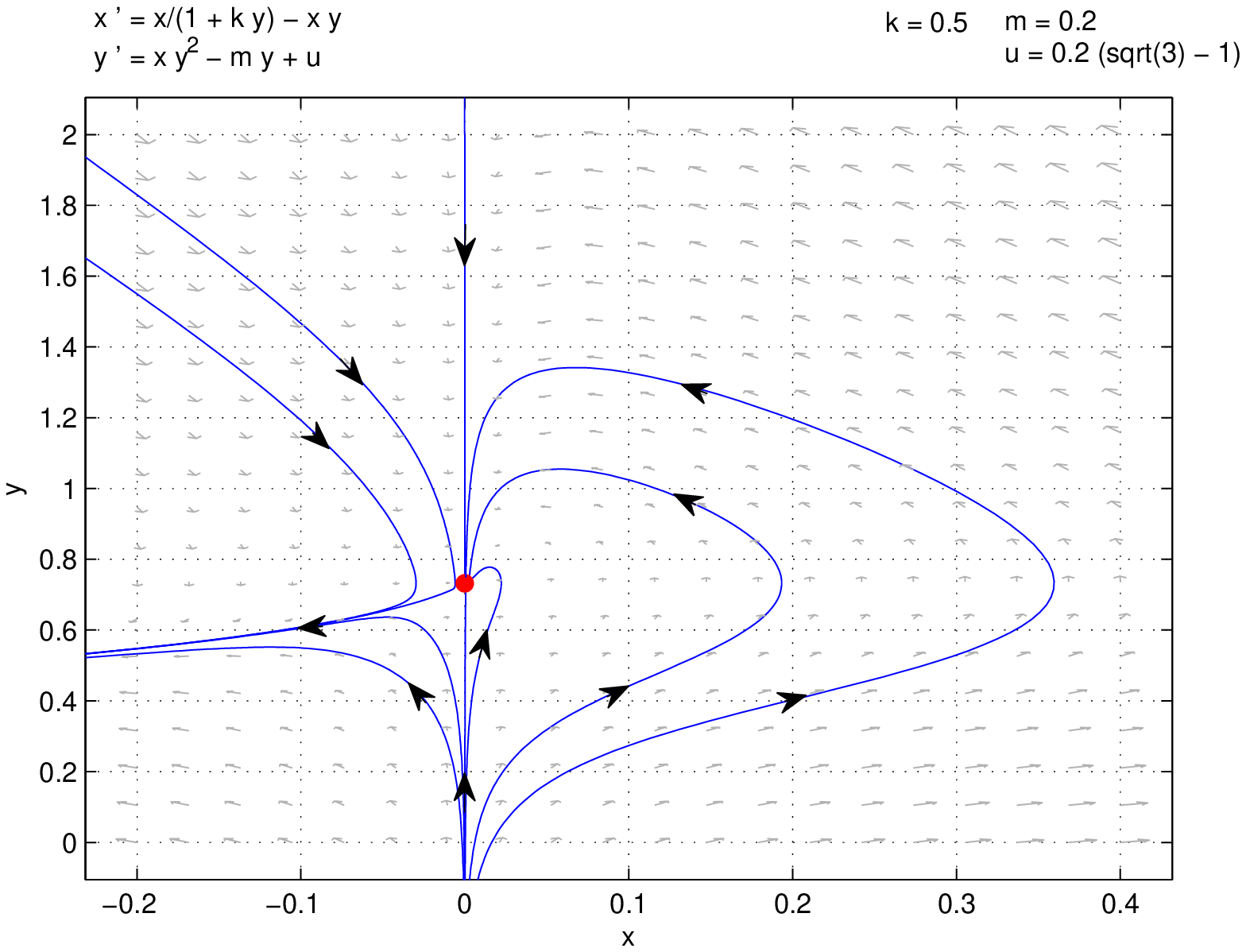}
\end{minipage}%
\begin{minipage}[c]{0.5\textwidth}
\centering
\includegraphics[height=5cm,width=6cm]{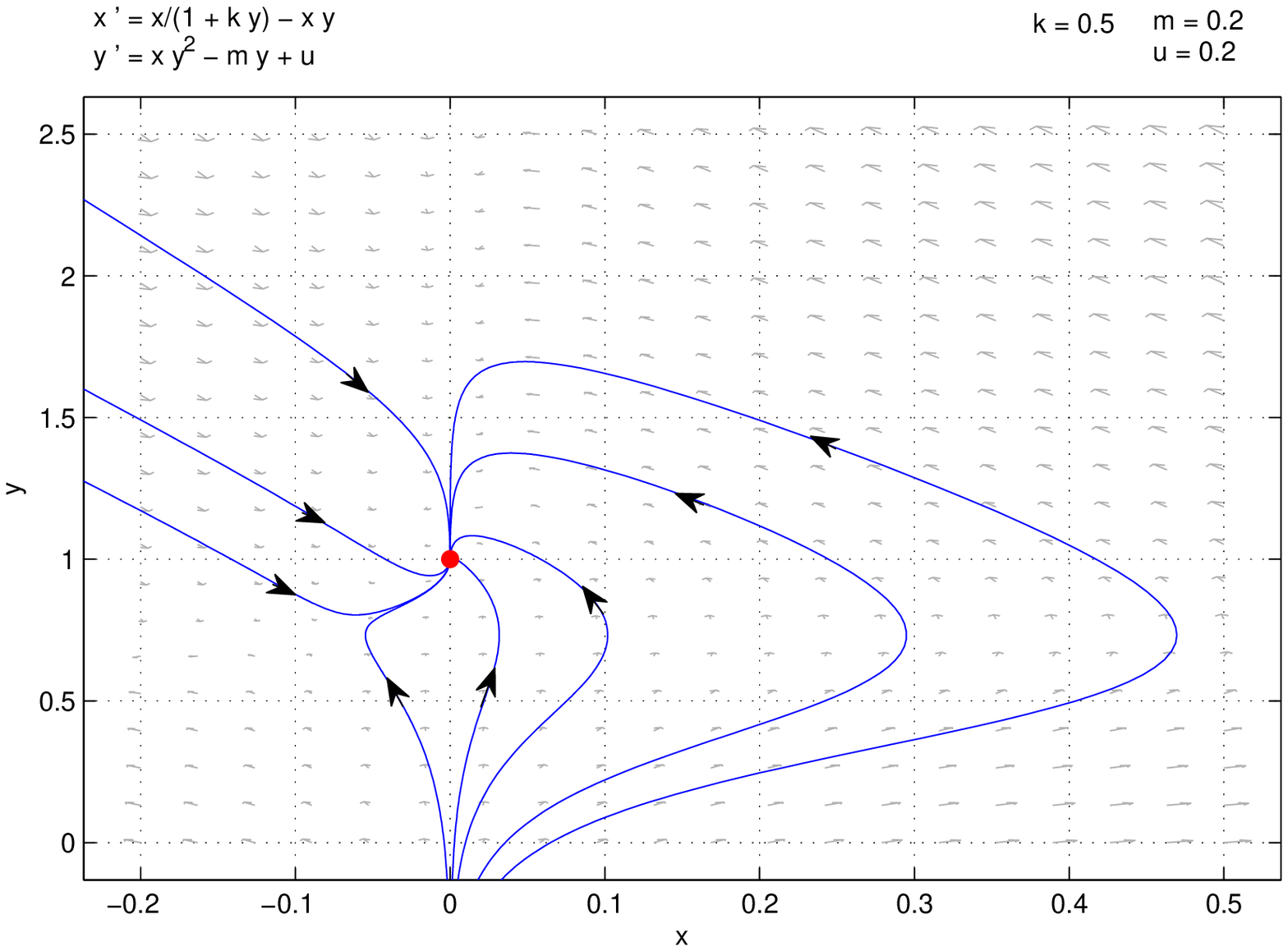}
\end{minipage}
(c) $u=u_0$
\ \ \  \ \ \ \ \ \ \ \ \ \
\ \ \  \ \ \ \ \ \ \ \ \ \
\ \ \  \ \ \ \ \ \ \ \ \ \
\ \ \  \ \ \ \ \ \ \ (d) $u>u_0$
\end{figure}
\centerline{Fig.1~~~Dynamic behavior of system \eqref{eq4.3} in the cases that $(1)-(4)$}

$(5)$ taking $u=0.1(\sqrt{3}-1)$, $u=\frac{u_0}{2}$, system \eqref{eq4.3} has two equilibria,
$E_2(0,y_2)=(0,5u)\approx(0,0.366)$ is a saddle, and $E_3=\left(\frac{0.2y_3-u}{y_3^2},y_3\right)\approx(0.137,0.732)$ is a center type stable focus.
In this situation, system \eqref{eq4.3} undergoes a Hopf bifurcation, and the Hopf bifurcation periodic solution is asymptotic stable. See Fig.2. Furthermore, from Fig.1(a),(b) and Fig.2, we can see the Hopf bifurcation is subcritical.

\begin{figure}[H]
\centering
\begin{minipage}[c]{0.5\textwidth}
\centering
\includegraphics[height=5cm,width=6cm]{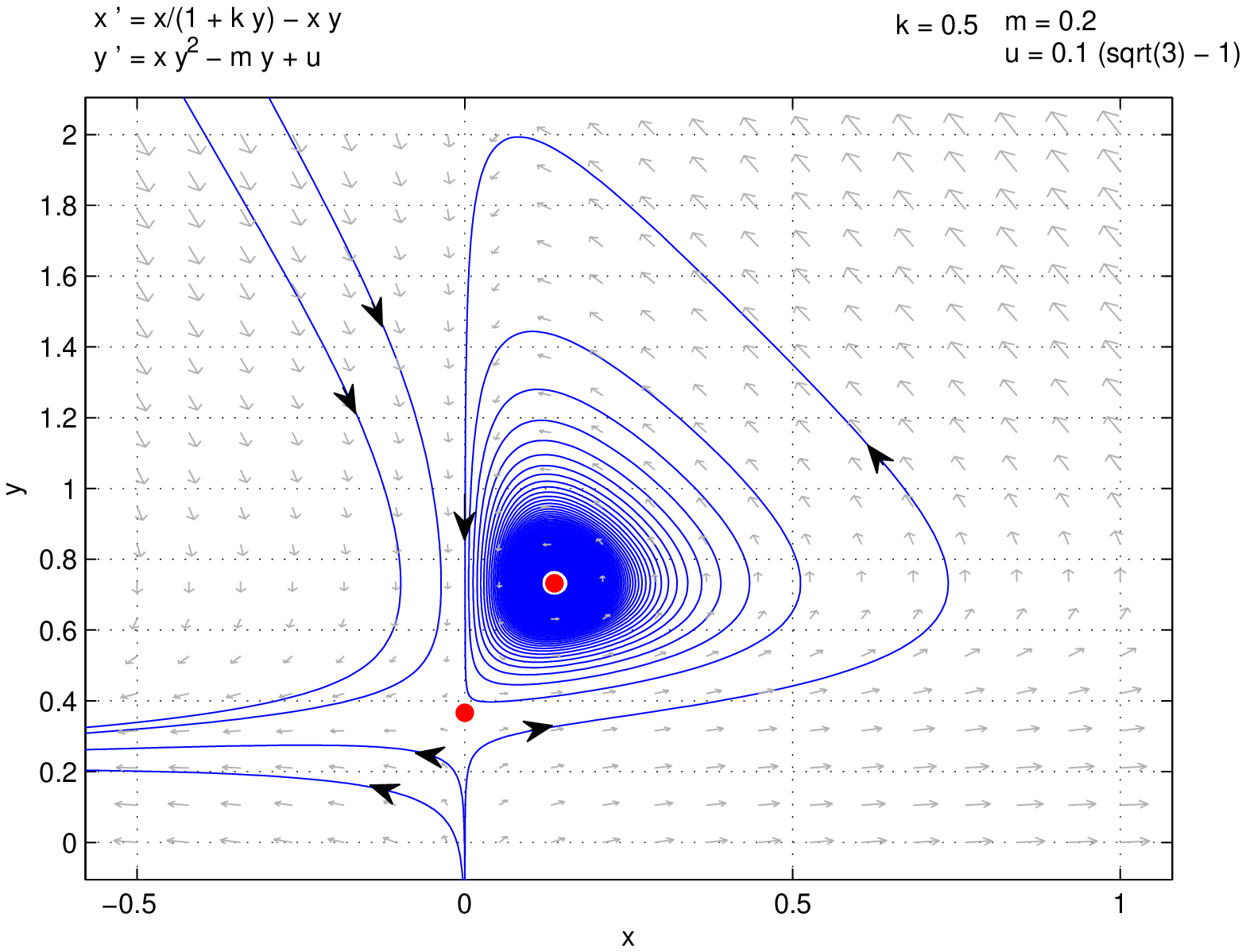}
\end{minipage}%
\begin{minipage}[c]{0.5\textwidth}
\centering
\includegraphics[height=5cm,width=6cm]{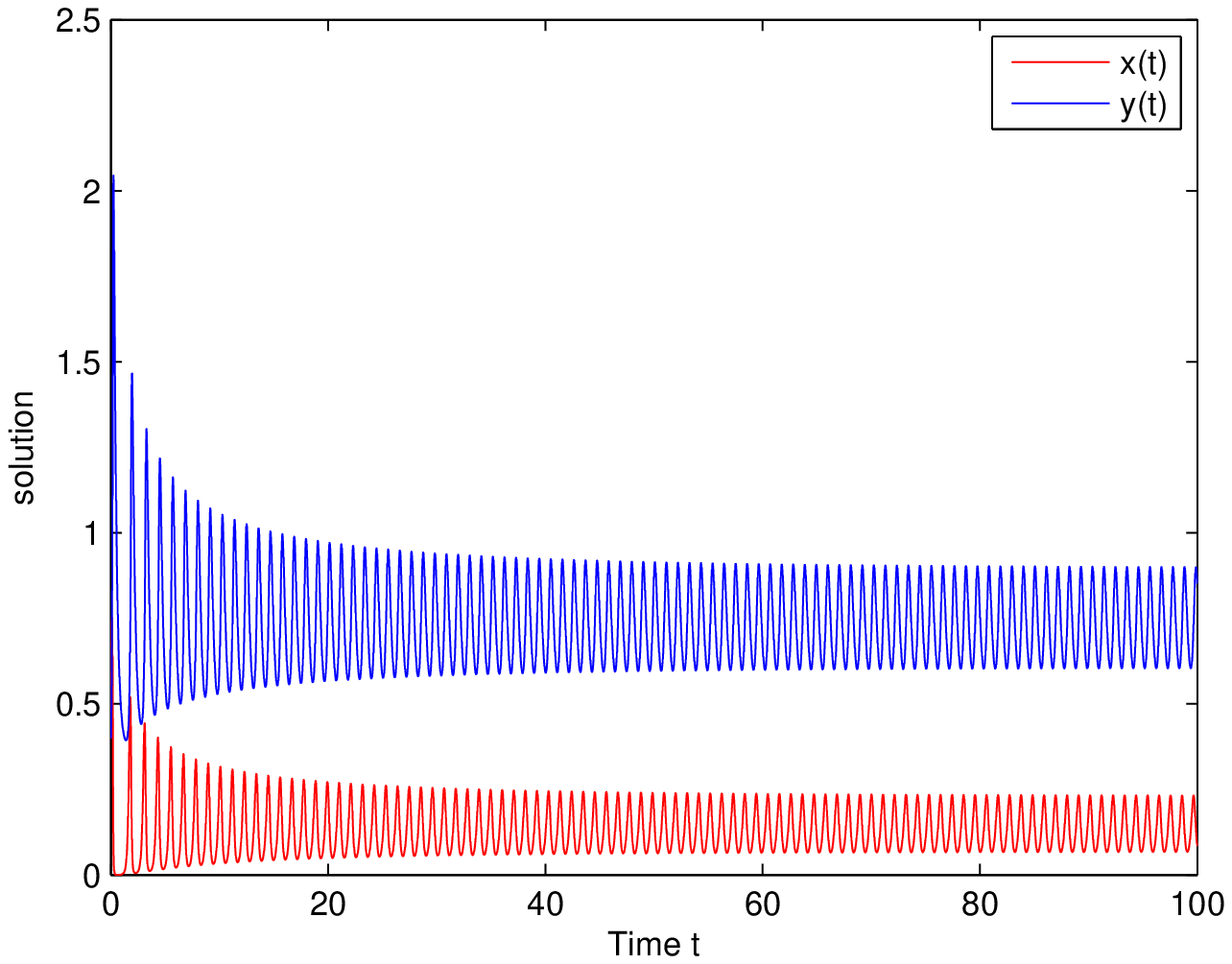}
\end{minipage}
\end{figure}
\centerline{Fig.2~~~The phase diagram of system \eqref{eq4.3} in the case that $u=\frac{u_0}{2}$ and the time series}
\centerline{diagram of the Hopf bifurcation periodic solution}

\end{Example}

\section{\Large{Conclusion}}


In order to observe the impact of the addition of inhibiting factor on model \eqref{eq3.2}, we compare the results of this paper with literature \cite{Wang2009}, it can be seen that the models both with inhibiting effect and without inhibiting effect have two equilibria \raisebox{0mm}{---} a pest-free equilibrium and a positive equilibrium. In order to intuitively compare the differences between the equilibria and stability of the two models, we give two tables in follows, and in where we use the symbols of this paper uniformly.

\begin{center}
 \makeatletter\def\@captype{table}\makeatother
 \newcommand{\tabincell}[2]{\begin{tabular}{@{}#1@{}}#2\end{tabular}}
 \caption{ Comparison at the pest-free equilibria}
 \small
   \label{TableAA}
  \centering
  \renewcommand\arraystretch{1.3}
  \setlength{\tabcolsep}{1mm}{
  \begin{tabular}{lll}
\hline
Equilibrium        & \hspace{4em}Existence &  \hspace{4em}Stability \\
\cline{1-3}
 Without inhibiting effect~~$(0,\frac{u}{m})$      & \hspace{4em}Always exists  &  \hspace{4em}\tabincell{l}{$u<m$, unstable \\$u>m$, stable } \\
\cline{1-3}
 With inhibiting effect~~$(0,\frac{u}{m})$      & \hspace{4em}Always exists &  \hspace{4em}\tabincell{l}{
 $u<my_3$, unstable \\$u>my_3$, stable } \\
\hline
\end{tabular}}
\end{center}

\begin{center}
 \makeatletter\def\@captype{table}\makeatother
 \newcommand{\tabincell}[2]{\begin{tabular}{@{}#1@{}}#2\end{tabular}}
 \caption{ Comparison at the positive equilibria}
 \small
   \label{TableAA}
  \centering
  \renewcommand\arraystretch{1.3}
  \setlength{\tabcolsep}{1mm}{
  \begin{tabular}{lll}
\hline
Equilibrium        & \hspace{4em}Existence &  \hspace{4em}Stability \\
 \cline{1-3}
 Without inhibiting effect~~$(m-u,1)$      & \hspace{4em}$u<m$  &  \hspace{4em}\tabincell{l}{$u<\frac{m}{2}$, unstable  \\$\frac{m}{2}<u<m$, stable } \\
\cline{1-3}
 With inhibiting effect~~$(\frac{my_3-u}{y_3^2},y_3)$      & \hspace{4em}$u<my_3$
 & \hspace{4em}\tabincell{l}{ $u<\frac{my_3}{2}$, unstable  \\$\frac{my_3}{2}<u<my_3$, stable } \\
\hline
\end{tabular}}
\end{center}



From the data in the above two tables, the addition of the inhibiting factor makes the boundary of equilibria change from $m$ to $my_3$ (from $\frac{m}{2}$ to $\frac{my_3}{2}$). It can be obtained by calculation that $y_3=\frac{\sqrt{1+4k}-1}{2k}$ decreases monotonically as $k$ increases and
$$\lim\limits_{k\rightarrow0}y_3=1,~~~0<y_3<1.$$ When $k=0$, that is, when model \eqref{eq3.2} has no inhibiting effect, it happens to be the model in literature \cite{Wang2009}, and the results are also consistent. As the level of inhibition $k$ increases, $y_3$ decreases, and model \eqref{eq3.2} can change from an unstable state to a stable state when $u$ is smaller. This suggests that the inhibition of nematodes on pests allows pests populations to be controlled with fewer nematodes released. And this is also consistent with reality.

Review this paper, we discussed the microbial pesticide model with inhibiting effect in the case that continuous release of nematodes. Through the analysis of the qualitative and stability of the model, we found the best solution to control pests. Next, we analyse system \eqref{eq1.2}, and from now on, we will no longer replace $\bar{y},\tau,\bar{k},\bar{m},\bar{u}$ with $y,t,k,m,u$. From the previous analysis and example verification, we can get the following conclusions:

$(1)$ Both the pest-free equilibrium and the positive equilibrium are unstable if $$u=\frac{c}{r^2}\bar{u}<\frac{c}{2r^2}u_0=\frac{c^2m}{4kr^4}\left(\sqrt{1+\frac{4kr}{c}}-1\right);$$

$(2)$ the pest-free equilibrium is unstable and the positive equilibrium is stable if $$\frac{c^2m}{4kr^4}\left(\sqrt{1+\frac{4kr}{c}}-1\right)=\frac{c}{2r^2}u_0\leq u<\frac{c}{r^2}u_0=\frac{c^2m}{2kr^4}\left(\sqrt{1+\frac{4kr}{c}}-1\right);$$

$(3)$ the unique equilibrium \raisebox{0mm}{---} pest-free equilibrium is stable if $$u\geq \frac{c}{r^2}u_0=\frac{c^2m}{2kr^4}\left(\sqrt{1+\frac{4kr}{c}}-1\right).$$

In summary, if we want to eliminate pests completely, we need to continuously release nematodes, and the speed is not less than $\frac{c^2m}{2kr^4}\left(\sqrt{1+\frac{4kr}{c}}-1\right)$. While if we only want to control the pest density within a certain range, then we only need to continuously release nematodes, and the speed is not less than $\frac{c^2m}{4kr^4}\left(\sqrt{1+\frac{4kr}{c}}-1\right)$.

\section{Conflict of Interest}
\hskip\parindent
The authors declare that they have no conflict of interest.

\section{Data Availability Statement}
My manuscript has no associated data. 

\section*{Contributions}
 We declare that all the authors have same contributions to this paper.

\end{document}